\documentclass[11pt]{amsart}
\usepackage{amssymb}
\usepackage{latexsym} 
\usepackage{srcltx} 
\usepackage{esint}
\usepackage{xcolor}

\def\nc{\newcommand}

\nc\pa{\partial}

\nc\CC{\mathbb{C}}
\nc\RR{\mathbb{R}}
\nc\QQ{\mathbb{Q}}
\nc\ZZ{\mathbb{Z}}
\nc\NN{\mathbb{N}}

\nc\m[1]{\left| #1\right|}
\nc\norm[1]{\left\| #1\right\|}
\newtheorem{theorem}{Theorem}[section]
\newtheorem{lemma}[theorem]{Lemma}

\newtheorem{remark}[theorem]{Remark}        
\numberwithin{equation}{section}

\begin{document}

	\title[On a capacitary strong type inequality]{On a capacitary strong type inequality and related capacitary estimates}

	\author[Keng Hao Ooi]
	{Keng Hao Ooi}
	\address{Department of Mathematics,
		Louisiana State University,
		303 Lockett Hall, Baton Rouge, LA 70803, USA.}
	\email{kooi1@lsu.edu}

	\author[Nguyen Cong Phuc]
	{Nguyen Cong Phuc}
	\address{Department of Mathematics,
		Louisiana State University,
		303 Lockett Hall, Baton Rouge, LA 70803, USA.}
	\email{pcnguyen@math.lsu.edu}

	\begin{abstract} We establish a  Maz'ya type  capacitary inequality which resolves a special case of a conjecture by David R. Adams. As a consequence, we obtain several equivalent norms for   Choquet integrals associated to Bessel or Riesz capacities.  This enables us to obtain bounds for the  Hardy-Littlewood maximal function in a sublinear setting.
	\end{abstract}

	\maketitle
	
	
	
\section{Introduction}

Let $\alpha$ be a real number and $s>1$. We define the space of Bessel potentials $H^{\alpha, s}=H^{\alpha, s}(\mathbb{R}^n)$, $n\geq 1$, 
as the completion of $C_c^\infty(\mathbb{R}^n)$ with respect to the norm 
$$\|u\|_{H^{\alpha,s}}=\|\mathcal{F}^{-1}[(1+|\xi|^2)^{\frac{\alpha}{2}}\mathcal{F}(u)] \|_{L^s(\mathbb{R}^n)},$$ 
where  $\mathcal{F}$ is the Fourier transform in $\mathbb{R}^n$. In the case $\alpha>0$, it follows  that (see, e.g., \cite{MH}) a function $u$ belongs to $H^{\alpha,s}$ if and only if
$$u= G_{\alpha}* f$$ 
for some $f\in L^s(\mathbb{R}^n)$, and moreover
$\|u\|_{H^{\alpha, s}}=\|f\|_{L^s(\mathbb{R}^n)}.$ 
Here $G_\alpha$ is  the Bessel kernel of order $\alpha$ defined by $G_{\alpha}(x):= \mathcal{F}^{-1}[(1+|\xi|^2)^{\frac{-\alpha}{2}}](x)$.

Recall that the  Bessel  capacity associated to the Bessel potential space 
$H^{\alpha, s}$ is defined for any set  $E\subset\mathbb{R}^n$ by 
\begin{equation*} 
{\rm Cap}_{\alpha, \,s}(E):=\inf\Big\{\|f\|_{L^s(\mathbb{R}^n)}^{s}: f\geq0, G_{\alpha}*f\geq 1 {\rm ~on~} E \Big\}.
\end{equation*}

A function $f: \RR^n\rightarrow [-\infty, +\infty]$ is said to be defined  quasieverywhere (q.e.) if it is defined  at every point of $\RR^n$ except for only a set of 
zero  capacity $\text{Cap}_{\alpha,s}$.
The notion of Choquet integral associated to Bessel  capacities will be important in this work. For a q.e. defined function $w:\RR^n \rightarrow [0,\infty]$, the Choquet integrals of $w$ is defined by 
\begin{equation*}
\int_{\RR^n} w\, d {\rm Cap}_{\alpha,s}:=\int_{0}^{\infty}\text{Cap}_{\alpha,s}(\{x\in\mathbb{R}^n: w(x)>t\})dt.
\end{equation*}

One of the fundamental results of  potential theory is the following Maz'ya's capacitary inequality, originally obtained by Maz'ya, and subsequently extended by Adams,  Dahlberg, and Hansson:
\begin{equation*} 
\int_{\RR^n} (G_\alpha * f)^s d{\rm Cap}_{\alpha,s} \leq  A \int_{\RR^n} f^s dx,
\end{equation*}
which holds for any nonnegative Lebesgue measurable function $f$.   See, e.g., \cite{Maz, MS, AH}, and in particular, see Section 2.3.1 and the historical comments in Section 2.3.13 of \cite{Maz}. This kind of  capacitary  inequalities and their many applications  are discussed in Chapters 2, 3 and 11 of \cite{Maz}.

In \cite{Ad2}, Adams conjectured (in the context of Riesz capacities and Riesz  potentials) that   another capacitary strong type inequality 
\begin{equation}\label{capstrong2}
\int_{\RR^n} (G_\alpha * f) d {\rm Cap}_{\alpha,s} \leq  A \int_{\RR^n} f^s (G_\alpha * f)^{1-s}  dx
\end{equation}
holds for any nonnegative Lebesgue measurable function $f$ (see Equ. (3.11) in \cite{Ad2}). (The integral $\int_{\RR^n} f^s (G_\alpha * f)^{1-s}dx$ is understood
as $\infty$ whenever $f=\infty$ on a set of positive Lebesgue measure. In the case $f\equiv 0$, it is understood as 0). Moreover, he essentially showed  for the corresponding Riesz capacities and   potentials that this is true provided $\alpha$ is an \emph{integer} in $(0, n)$ (see page 23 in \cite{Ad2}). However, we observe that his argument does not appear to work for Bessel  capacities and Bessel  potentials as in \eqref{capstrong2} even with  integers $\alpha\in (0,n)$.

One of the main purposes of this note is to verify \eqref{capstrong2} for any real $\alpha>0$. 
\begin{theorem}\label{AdamsCon}
	Let $\alpha>0$ and $s>1$ be such that $\alpha s\leq n$. There exists a constant $A>0$ such that \eqref{capstrong2} holds for any nonnegative Lebesgue measurable function $f$.
\end{theorem}

Our proof of \eqref{capstrong2} is also applicable to the setting of 
Riesz capacities and   potentials, and thereby extends the above mentioned results of \cite{Ad2} to all real $\alpha\in(0,n)$.

Our approach to \eqref{capstrong2} is based mainly in our recent work \cite{OP} in which predual spaces to a Sobolev multiplier type space were considered. 
For $\alpha>0, s>1,$ and $p> 1$, let  $M^{\alpha,s}_p=M^{\alpha,s}_p(\mathbb{R}^n)$ be the Banach space of  functions $f\in L^p_{\rm loc}(\mathbb{R}^n)$
such that the trace inequality 
\begin{equation}\label{traceG}
\left(\int_{\mathbb{R}^n} (G_\alpha*h) ^s |f|^p dx\right)^{\frac{1}{p}} \leq A \|h\|_{L^s(\RR^n)}^{\frac{s}{p}}
\end{equation} 
holds for all nonnegative $h\in L^s(\RR^n)$. A norm of a function $f\in M^{\alpha,s}_p$ can be defined as
\begin{align}\label{sobolevN}
\|f\|_{M^{\alpha,s}_p}:= \sup_{K}\left(\frac{\int_{K}|f(x)|^{p}dx}{\text{Cap}_{\alpha,s}(K)}\right)^{1/p},
\end{align}
where the supremum is taken over all compact sets $K\subset{\mathbb{R}}^{n}$ with non-zero capacity. Note that the right-hand side of \eqref{sobolevN}  is known to be equivalent to the least possible constant $A$ in \eqref{traceG} (see \cite{MS, AH}).

In \cite{OP}, we showed that  a predual of $M^{\alpha,s}_p$ is its K\"othe dual space $(M^{\alpha,s}_p)'$ defined by
\begin{equation*}
(M^{\alpha,s}_p)'=\left\{{\rm measurable~ functions~} f: \sup\int|fg|dx<+\infty\right\},
\end{equation*}
where the supremum is taken over all functions $g$ in the unit ball of $M^{\alpha,s}_p$. The norm of $f\in (M^{\alpha,s}_p)'$ is defined as the above 
supremum. Thus we have 
$$[(M^{\alpha,s}_p)']^*=M^{\alpha,s}_p,$$
with equality of norms. Various characterizations of $(M^{\alpha,s}_p)'$ can be found in \cite{OP}. For our purpose here the case $p=s'=s/(s-1)$ is of special interest.  
In particular, as mentioned in Remark 2.10 in \cite{OP}, it follows from \cite{MV, KV} that  the space $M^{\alpha,s}_{s'}$ is an intrinsic space associated to the nonlinear integral equation
$$u= G_\alpha*(u^{s'}) +   f \qquad \text{a.e.}$$


Another important observation in \cite{OP} is the following equivalence:
\begin{equation}\label{funequN}
\int_{\RR^n} |u| d {\rm Cap}_{\alpha,s} \simeq \gamma_{\alpha,s}(u),
\end{equation}
which holds for all q.e. defined functions $u$ in $\RR^n$. Here the functional $\gamma_{\alpha,s}(\cdot)$ is defined for each q.e. defined function $u$ by
$$\gamma_{\alpha,s}(u):=\inf\left\{\int f^sdx: 0\leq f\in L^s(\RR^n) \text{ and } G_\alpha*f\geq |u|^{\frac{1}{s}} \text{ q.e.} \right\}.$$
Note that $\gamma_{\alpha,s}(tu)=|t|\gamma_{\alpha,s}(u)$ for all $t\in\RR$ and moreover $\gamma_{\alpha,s}(u_1+u_2)\leq \gamma_{\alpha,s}(u_1) +\gamma_{\alpha,s}(u_2)$ (see \cite{OP}). 
On the other hand, the Choquet integral $\int_{\RR^n} |\cdot| d {\rm Cap}_{\alpha,s}$ is known to be subadditive only for $s=2$ and $0<\alpha\leq 1$. In particular, the set of all q.e. defined functions 
$u$ in $\RR^n$ such that $\int_{\RR^n} |u| d {\rm Cap}_{\alpha,s}<+\infty$ is a normable space. An argument as in the proof of Proposition 2.3 in \cite{OP} can be used to show that this space is complete.

As a consequence of  \eqref{funequN} and the proof of Theorem \ref{AdamsCon}, in this paper we obtain two other characterizations for the Choquet integral.
For a q.e. defined function $u$ in $\RR^n$ we denote by $\lambda_{\alpha,s}(u)$ and $\beta_{\alpha,s}$, $\alpha>0, s>1$, the following quantities:
$$\lambda_{\alpha,s}(u):= \inf\left\{\norm{f}_{(M^{\alpha,s}_{s'})'}:  0\leq f\in (M^{\alpha,s}_{s'})' \text{ and }  G_\alpha*f\geq |u| \text{ q.e.} \right\}$$\
and
$$\beta_{\alpha,s}(u):= \inf\left\{\int_{\RR^n} f^s (G_\alpha*f)^{1-s}dx:  f\geq0,\,   G_\alpha*f\geq |u| \text{ q.e.} \right\}.$$

\begin{theorem}\label{Newnorm2} Let $\alpha>0$ and $s>1$ be such that $\alpha s\leq n$. For any q.e. defined function  $u$ in $\RR^n$ it holds that 
	\begin{equation}\label{funequNew}
	\int_{\RR^n} |u| d {\rm Cap}_{\alpha,s} \simeq \lambda_{\alpha,s}(u) \simeq \beta_{\alpha,s}(u).
	\end{equation}
	In particular, we have 
	$$
	{\rm Cap}_{\alpha,s}(E) \simeq \lambda_{\alpha,s}(\chi_{E}) \simeq \beta_{\alpha,s}(\chi_E)
	$$
	for any set $E\subset\RR^n$.
\end{theorem}

To discuss a consequence of Theorem \ref{Newnorm2}, we now  recall that the (center) local Hardy-Littlewood maximal function is 
defined for each $f\in L^1_{\rm loc}(\mathbb{R}^n)$ by 
\begin{equation*}
{\bf M}^{\rm loc} f (x)= \sup_{0<r\leq 1}   \frac{1}{|B_r(x)|} \int_{B_{r}(x)} |f(y)|dy.
\end{equation*} 
for every $x\in\RR^n$.

\begin{theorem}\label{Maximal} Let $\alpha>0$ and $s>1$ be such that $\alpha s\leq n$.  For any $q> (n-\alpha)/n$ and any measurable and q.e. defined  function $f$, we have 
	$$\int_{\RR^n} ({\bf M}^{\rm loc} f)^{q} d {\rm Cap}_{\alpha,s} \leq A(n,\alpha,s,q) \int_{\RR^n} |f|^q d {\rm Cap}_{\alpha,s}.$$	
\end{theorem}

An interesting aspect of Theorem \ref{Maximal} is that the power $q$ is allowed to be strictly less than $1$. Moreover, here we do not assume any continuity assumption 
on $f$. See \cite{Ad1}, Theorem 7.5 in \cite{AX}, and \cite{OV} for some related results.

Finally, we remark that Theorems \ref{AdamsCon}, \ref{Newnorm2}, and \ref{Maximal} also hold in the homogeneous setting provided $\alpha\in(0,n)$, $s>1$, and Bessel potentials and capacities are replaced by  the corresponding  Riesz potentials and capacities.  Moreover, in the homogeneous setting the local Hardy-Littlewood maximal function 
${\bf M}^{\rm loc}$ can be
replaced by the larger standard   Hardy-Littlewood maximal function.

Recall that the Riesz kernel $I_{\alpha}$, $\alpha\in (0,n)$, is  defined as the inverse Fourier transform of $|\xi|^{\alpha}$ (in the distributional sense), and explicitly we have 
$I_\alpha(x)= \gamma(n,\alpha) |x|^{\alpha-n}$, where $\gamma(n, \alpha)=\Gamma(\tfrac{n-\alpha}{2})/[\pi^{n/2} 2^\alpha \Gamma(\tfrac{\alpha}{2})]$. 
The Riesz potential of a nonnegative measure $\mu$ is defined by the convolution $I_\alpha*\mu$. For $\alpha\in(0,n)$ and $s>1$, the Riesz capacity ${\rm cap}_{\alpha,s}$ is defined  
for each set $E\subset\RR^n$ by 
\begin{equation*}
{\rm cap}_{\alpha, \,s}(E):=\inf\Big\{\|f\|_{L^{s}(\RR^n)}^{s}: f\geq0, I_{\alpha}*f\geq 1 {\rm ~on~} E \Big\}.
\end{equation*}

This capacity is the capacity associated to the homogeneous Sobolev space $\dot{H}^{\alpha,s}$ (see Section 9 in \cite{OP}).

\vspace{.2in}
\noindent {\bf Notation.}  The characteristic function of a set $E\subset\RR^n$ is denoted by $\chi_{E}$.
For two quantities $A$ and $B$, we write $A\simeq B$ to mean that there exist positive constants $c_1$ and $c_2$ such that $c_1 A\leq B\leq c_2 A$.

\section{Proof of Theorem \ref{AdamsCon}}
\begin{proof}[Proof of Theorem \ref{AdamsCon}] 
	Let $L^1(C)$ denote the space of quasicontinuous function $f$ in $\RR^n$ such that 
	$$\norm{f}_{L^1(C)}:=\int_{\RR^n} |f| d{\rm Cap}_{\alpha,s}<+\infty.$$ 
	Recall a function $f$ is said to be quasicontinuous
	(with respect to  ${\rm Cap}_{\alpha, s}$) if for any $\epsilon>0$ there exists an open set $O$ such that ${\rm Cap}_{\alpha,s}(O)<\epsilon$ and $f$ is continuous in $O^c:=\mathbb{R}^n\setminus O$. It is known that the dual of $L^1(C)$ can be identify with the space 
	$\mathfrak{M}^{\alpha,s}=\mathfrak{M}^{\alpha,s}(\mathbb{R}^n)$ which consists of   locally finite signed measures $\mu$ in $\mathbb{R}^n$ such that the norm $\|\mu\|_{\mathfrak{M}^{\alpha,s}}<+\infty$ (see Theorem 2.4 in \cite{OP}). Here  we define 
	\begin{align*}
	\|\mu\|_{\mathfrak{M}^{\alpha,s}}:=\sup_{K}\dfrac{|\mu|(K)}{\text{Cap}_{\alpha,s}(K)},
	\end{align*}
	where the supremum is taken over all compact sets $K\subset{\mathbb{R}}^{n}$ such that $\text{Cap}_{\alpha,s}(K)\not=0$.
	
	In view of \eqref{funequN}, $L^1(C)$ is normable and thus it follows from Hahn-Banach Theorem that  for any $u\in L^1(C)$ we have 
	\begin{equation}\label{HB}
	\norm{u}_{L^1(C)} \simeq \sup \left\{\left|\int u d\mu\right|: \|\mu\|_{\mathfrak{M}^{\alpha,s}}\leq 1\right\}.
	\end{equation}
	
	Let $f$ be a nonnegative measurable and bounded function with compact support. Applying  \eqref{HB} with $u=G_\alpha* f$ we have 
	\begin{align*}
	\int_{\RR^n} G_\alpha* f d {\rm Cap}_{\alpha,s} &\leq A \sup_{\|\mu\|_{\mathfrak{M}^{\alpha,s}}\leq 1} \int G_\alpha* f d|\mu|\\
	&= A \sup_{\|\mu\|_{\mathfrak{M}^{\alpha,s}}\leq 1} \int (G_\alpha* |\mu|) f dx\\
	&\leq A \norm{f}_{(M^{\alpha, s}_{s'})'}  \sup_{\|\mu\|_{\mathfrak{M}^{\alpha,s}}\leq 1}  \norm{G_\alpha* |\mu|}_{M^{\alpha, s}_{s'}} \\
	&\leq A \norm{f}_{(M^{\alpha, s}_{s'})'},
	\end{align*}
	where the last inequality follows from Theorem 1.2 in \cite{MV}. By density (see Remark 3.3 in \cite{OP}) we see that the inequality 
	\begin{equation}\label{GfL1C}
	\int_{\RR^n} G_\alpha* f d {\rm Cap}_{\alpha,s} \leq A \norm{f}_{(M^{\alpha, s}_{s'})'}
	\end{equation}
	holds for any nonnegative function $f\in (M^{\alpha, s}_{s'})'$.

	In proving \eqref{capstrong2} we may assume that $\int_{\RR^n} f^s (G_\alpha * f)^{1-s} dx<+\infty$ and hence $f$ is finite a.e. by our convention. In this case we must have that  $f\in (M^{\alpha, s}_{s'})'$.
	Indeed, for any $g\in M^{\alpha, s}_{s'}$ such that $\|g\|_{M^{\alpha, s}_{s'}}\leq 1$  by  Remark 2.10 in \cite{OP} and \cite{KV}, there exists a nonnegative  function 
	$u\in L^{s'}_{\rm loc}(\RR^n)$ such that
	$$u= G_\alpha*(u^{s'}) +   \frac{|g|}{ M} \qquad \text{a.e.}$$ 
	for a constant $M>0$ independent of $g$ and $u$. Thus, as in \cite{BP} (see also \cite{KV}), we have
	\begin{align}
	\label{intfg}\int_{\RR^n} f |g| dx &= M  \int_{\RR^n} f( u-G_\alpha*(u^{s'}))  dx\\
	&= M  \int_{\RR^n} (f u- u^{s'}G_\alpha*f)  dx\nonumber\\
	&= M  \int_{\RR^n} G_\alpha*f\left( u \frac{f}{G_\alpha*f}- u^{s'}\right)  dx \nonumber\\
	&\leq M s^{-s}(s-1)^{s-1}  \int_{\RR^n} f^s (G_\alpha*f)^{1-s}  dx, \nonumber
	\end{align} 
	where we used the Young's inequality $ab - a^{s'}/s' \leq b^{s}/s$, $a,b\geq 0$, in the last inequality. Thus taking the supremum over 
	$g\in M^{\alpha, s}_{s'}$ such that $\|g\|_{M^{\alpha, s}_{s'}}\leq 1$ in \eqref{intfg}, we find
	\begin{equation}\label{mvn}
	\|f\|_{(M^{\alpha, s}_{s'})'}\leq A \int_{\RR^n} f^s (G_\alpha*f)^{1-s}  dx <+\infty.
	\end{equation}
	
	Finally, combining \eqref{GfL1C} with \eqref{mvn} we obtain \eqref{capstrong2} as desired.
\end{proof}	

\begin{remark} We remark that   \eqref{capstrong2} and \eqref{GfL1C} are indeed equivalent.  On one hand, the proof above shows that  \eqref{GfL1C} implies \eqref{capstrong2}. On the other hand, \eqref{capstrong2} implies that 
	\begin{equation*}
	\int_{\RR^n} G_\alpha* f d {\rm Cap}_{\alpha,s} \leq A \norm{f}_{KV}
	\end{equation*}
	for any nonnegative measurable function $f$. Here we define
	$$\norm{f}_{KV}:=\inf\left\{\int_{\RR^n} h^s (G_\alpha*h)^{1-s}  dx: h\geq |f| \text{ a.e.} \right\}.$$
	($\norm{f}_{KV}$ is understood as $\infty$ if there is no  measurable function $h$ such that $h\geq |f|$ a.e. and 
	$\int_{\RR^n} h^s (G_\alpha*h)^{1-s}  dx<+\infty$.) As we observe in Remark 2.10 in \cite{OP},  the  two-sided bound $\|f\|_{(M^{\alpha, s}_{s'})'} \simeq \norm{f}_{KV}$
	follows from   \cite{KV, MV}. Thus  \eqref{capstrong2}  implies \eqref{GfL1C}.  
\end{remark}

%
%
%
%

\section{Proof of Theorem \ref{Newnorm2}}

In order to prove Theorem \ref{Newnorm2}, we first prove the following ``integration by parts" lemma.

\begin{lemma}\label{IBPL} Let $\alpha>0, s>1$ be such that $\alpha s\leq n$. Suppose that $\mu$ is a nonnegative measure such that the diameter of ${\rm supp}(\mu)$ is less than $1$. Then there is a constant $A=A(n,\alpha, s)>0$ such that, for $f=(G_{\alpha}*\mu)^{s'-1}$, we have 
	\begin{equation*}
	(G_{\alpha}*f)^s \leq A G_{\alpha}*[f (G_{\alpha}*f)^{s-1}] 
	\end{equation*}	 
	pointwise everywhere in $\RR^n$.
\end{lemma}

\begin{remark} For Riesz potentials, this lemma has been established for all $f\geq 0$ in \cite{VW} (see also \cite{KV, Ver}). In our setting, which deals with  Bessel potentials, it is necessary to require $\mu$
	to have compact support.	
\end{remark}

\begin{proof}[Proof of Lemma \ref{IBPL}]
	Without loss of generality, we may assume that ${\rm supp}(\mu)\subset B_{1/2}(0)$.
	With $f=(G_{\alpha}*\mu)^{s'-1}$, we write $f=f_1+f_2$, where
	$$f_1= f\chi_{B_3(0)} \quad \text{and} \quad f_2= f\chi_{B_3(0)^c} \quad (B_3(0)^c=\RR^n\setminus B_3(0)).$$
	
	Then 
	\begin{equation}\label{split}
	(G_\alpha*f)^s\leq A [(G_\alpha*f_1)^s + (G_\alpha*f_2)^s].
	\end{equation}
	
	We shall use the following pointwise two-sided estimates for $G_\alpha$ (see, e.g., Section 1.2.4 in \cite{AH}):
	\begin{equation}\label{Ga1}
	G_\alpha(x) \simeq |x|^{\alpha -n}, \quad \forall |x|\leq 15, (0<\alpha<n).
	\end{equation}
	and 
	\begin{equation}\label{Ga2}
	G_\alpha(x) \simeq G_\alpha(x+y), \quad \forall |x|\geq 3, |y|\leq 1, (\alpha>0).
	\end{equation}
	
	We mention that  \eqref{Ga2} follows from the asymptotic behavior $G_\alpha$ near infinity that can be found, e.g., in Equ. (1.2.24) in \cite{AH}. 
	
	We now write
	\begin{equation*}
	[G_\alpha*f_1(x)]^s =\int_{|y|\leq 3} G_\alpha(x-y) f(y)\left[ \int_{|z|\leq 3} G_\alpha(x-z) f(z)dz\right]^{s-1} dy. 
	\end{equation*}
	
	Thus if $|x|\geq 10$, then $|x-z|\geq 7\geq  |y-z|$, which yields that 
	$$G_\alpha(x-z)\leq G_\alpha(y-z).$$
	
	Therefore, we get 
	\begin{equation*}
	[G_\alpha*f_1(x)]^s \leq G_{\alpha}*[f (G_{\alpha}*f)^{s-1}](x)
	\end{equation*}
	in the case $|x|\geq 10$.
	
	On the other hand, if $|x|<10$, then for $|y|\leq 3$ by \eqref{Ga1} we have
	$$ G_\alpha(x-y) \simeq |x-y|^{\alpha -n}.$$
	
	Thus applying Lemma 2.1 in \cite{Ver} we obtain
	\begin{equation*}
	[G_\alpha*f_1(x)]^s \leq A G_{\alpha}*[f_1 (G_{\alpha}*f_1)^{s-1}](x)\leq A G_{\alpha}*[f (G_{\alpha}*f)^{s-1}](x)
	\end{equation*}
	in the case $|x|<10$.
	
	Combining these two estimates we get that 
	\begin{equation}\label{Gf1}
	[G_\alpha*f_1(x)]^s \leq  A G_{\alpha}*[f (G_{\alpha}*f)^{s-1}](x), \quad \forall x\in \RR^n.
	\end{equation}

	To estimate $[G_\alpha*f_2(x)]^s$ we first observe the following bound
	\begin{equation}\label{Gf21}
	f_2(x)\leq A G_\alpha*f(x), \quad \forall x\in \RR^n.
	\end{equation}
	
	Inequality \eqref{Gf21} is trivial when $|x|<3$. On the other hand, for $|x|\geq 3$, we have by \eqref{Ga2},
	\begin{align*}
	(f_2(x))^{s-1}&=\int_{|y|<1/2} G_\alpha(x-y) d\mu(y)\leq A \int_{|y|<1/2} G_\alpha(x) d\mu(y) \\
	& = A \norm{\mu} G_\alpha(x).
	\end{align*}
	
	Note that for $|y-x|<1/2$ and $|x|\geq 3$, by \eqref{Ga2} we have 
	$$f(y)^{s-1}=\int_{|z|<1/2} G_\alpha(y-z) d\mu(z)\geq c_0\, G_\alpha(x) \norm{\mu},$$
	and so, for $|x|\geq 3$,
	\begin{align*}
	G_\alpha*f(x) &\geq \int_{|y-x|<1/2} G_\alpha(x-y) f(y) dy\\
	&\geq \int_{|y-x|<1/2} G_\alpha(x-y)  (c_0\, G_\alpha(x) \norm{\mu})^{s'-1} dy\\
	&\geq c\, (\norm{\mu} G_\alpha(x))^{s'-1} \geq c_1 f_2(x).
	\end{align*}
	
	Thus  \eqref{Gf21} is verified.
	Now by H\"older's inequality and \eqref{Gf21} we have 
	\begin{equation}\label{Gf22}
	[G_\alpha*f_2]^s \leq A G_\alpha*(f_2^s) \leq A_1 G_\alpha*[f (G_\alpha *f)^{s-1}].
	\end{equation}

	At this point, combining   \eqref{split}, \eqref{Gf1}, and \eqref{Gf22}, we obtain  the lemma.
\end{proof}

We are now ready to prove Theorem \ref{funequNew}.

\begin{proof}[Proof of Theorem \ref{funequNew}]
	Let $u$ be a q.e. defined function in $\RR^n$. Suppose that $f$ is a nonnegative measurable function such that $G_\alpha*f\geq |u|$ quasi-everywhere. Then by \eqref{GfL1C} and \eqref{mvn} it follows that
	\begin{align*}
	&\int_{\RR^n} |u| d {\rm Cap}_{\alpha,s} \leq \int_{\RR^n} G_\alpha * f d {\rm Cap}_{\alpha,s} \\
	& \leq A_1 \norm{f}_{(M^{\alpha, s}_{s'})'} \leq A_2 \int_{\RR^n} f^s (G_\alpha*f)^{1-s} dx. 
	\end{align*}
	Now taking the infimum over such $f$ we arrive at
	$$\int_{\RR^n} |u| d {\rm Cap}_{\alpha,s} \lesssim \lambda_{\alpha,s}(u) \lesssim \beta_{\alpha,s}(u).$$
	
	Thus to complete the proof, it is left to show that 
	\begin{equation}\label{nonlinearnorm}
	\beta_{\alpha,s}(u) \lesssim \int_{\RR^n} |u| d {\rm Cap}_{\alpha,s}.
	\end{equation} 
	
	To this end, we first show \eqref{nonlinearnorm} for $u=\chi_{E}$, where  $E$ is any set  such that ${\rm Cap}_{\alpha,s}(E)>0$ and the diameter of $E$ is less than $1$.
	By Theorems 2.5.6 and 2.6.3 in \cite{AH}
	one can find a nonnegative measure $\mu=\mu^{E}$ with ${\rm supp}(\mu) \subset \overline{E}$ (called capacitary measure for $E$) such that the function 
	$V^{E}=G_\alpha*((G_\alpha*\mu)^{s'-1})$ satisfies the following properties:
	
	\begin{equation*}
	\mu^E(\overline{E})={\rm Cap}_{\alpha,s}(E)=\int_{\RR^n} V^E d\mu^E= \int_{\RR^n} (G_\alpha *\mu^{E})^{s'} dx,
	\end{equation*}
	and
	\begin{equation*}
	V^E\geq 1 \quad \text{quasieverywhere on }E.
	\end{equation*}

	Let $f= (G_\alpha*\mu)^{s'-1}$. By Lemma  \ref{IBPL}, we have 
	$$\chi_{E}\leq (V^E)^s =(G_{\alpha}*{f})^s\leq A G_{\alpha}*[f (G_{\alpha}*f)^{s-1}] \quad \text{q.e.}$$

	Thus,
	\begin{align*}
	\beta_{\alpha,s}(\chi_E) &\leq A \int_{\RR^n} f^s (G_{\alpha}*f)^{(s-1)s}  \left\{G_{\alpha}*[f (G_{\alpha}*f)^{s-1}]\right\}^{1-s}dx\\
	& \leq A \int_{\RR^n} f^s (G_{\alpha}*f)^{(s-1)s} (G_{\alpha}*f)^{(1-s)s}dx\\
	& = A \int_{\RR^n} f^s dx= A \int_{\RR^n}  (G_\alpha*\mu)^{s'}dx =A\, {\rm Cap}_{\alpha,s}(E), 
	\end{align*}
	as desired.

	We now let $\{\mathcal{B}^j\}_{j\geq0}$ be a covering of $\RR^n$ by open balls with unit diameter. This covering is chosen in such a way that it has a finite multiplicity 
	depending only on $n$. We shall use the following quasi-additivity of ${\rm Cap}_{\alpha,s}$:
	\begin{equation}\label{quasi-add}
	\sum_{j\geq0} {\rm Cap}_{\alpha,s}(E\cap \mathcal{B}^j) \leq M {\rm Cap}_{\alpha,s}(E)
	\end{equation}
	for any set $E\subset\RR^n$. For compact sets $E$, a proof of \eqref{quasi-add} can be found in Proposition 3.1.5 in \cite{MS}. The same proof also works for any set $E$ provided one uses Corollary 2.6.8 in \cite{AH}.
	
	In proving \eqref{nonlinearnorm} we may assume that $\int_{\RR^n} |u|d {\rm Cap}_{\alpha,s}<+\infty$.
	Let $E_k=\{2^{k-1}<|u|\leq 2^{k}\}$ and $E_{j,k}=E_k\cap \mathcal{B}^j$ for $k\in \ZZ$ and $j\geq 0$.  
	We have 
	\begin{equation}\label{betasum}
	\beta_{\alpha,s}(u)=\beta_{\alpha,s}\left(\sum_{k\in \ZZ} |u| \chi_{E_{k}}\right)\leq  \beta_{\alpha,s}\left(\sum_{k\in \ZZ}\sum_{j\geq 0} |u| \chi_{E_{j,k}}\right).
	\end{equation}
	
	For $k\in\ZZ$ and $j\geq0$, let 
	$$f_{j,k}= (G_\alpha*\mu^{E_{j,k}})^{s'-1} \quad \text{and} \quad F_{j,k}= f_{j,k} (G_\alpha*f_{j,k})^{s-1}.$$
	
	By the above argument, we have 
	$$G_\alpha *(2^k F_{jk})\geq c\, |u|\chi_{E_{j,k}}\quad \text{q.e.}$$
	and 
	$$\int_{\RR^n} (2^k F_{jk})^s (G_\alpha* (2^k F_{j,k}))^{1-s} dx \leq A 2^k {\rm Cap}_{\alpha,s}(E_{jk}).$$
	
	By \eqref{mvn}, this gives 
	\begin{equation}\label{FjkMpr}
	\norm{2^k F_{j,k}}_{(M^{\alpha,s}_{s'})'}\leq A 2^k {\rm Cap}_{\alpha,s}(E_{jk}).
	\end{equation}
	
	Set $F=\sup_{j,k} 2^k F_{j,k}$.  Then we have 
	$(G_\alpha*F)^{1-s}\leq (G_\alpha*(2^k F_{j,k}))^{1-s}$ for any $k\in\ZZ$ and $j\geq 0$. Moreover,  
	$$G_{\alpha}*F\geq c\, \sum_{k\in \ZZ} |u| \chi_{E_{k}} \geq  c_1\, \sum_{k\in \ZZ} \sum_{j\geq 0}|u| \chi_{E_{j,k}}\qquad \text{q.e.}$$
	due to the finite multiplicity of $\{\mathcal{B}^j\}_{j\geq 0}$.
	Also, it follows from  \eqref{quasi-add}  and \eqref{FjkMpr} that 
	\begin{align*}
	\norm{F}_{(M^{\alpha,s}_{s'})'}&\leq C_1 \sum_{k\in\ZZ} \sum_{j\geq 0}2^k {\rm Cap}_{\alpha,s}(E_{j,k}) \leq C_2 \sum_{k\in\ZZ} 2^k {\rm Cap}_{\alpha,s}(E_{k}) \\
	& \leq A \int_{\RR^n} |u| d {\rm Cap}_{\alpha,s} <+\infty.
	\end{align*}
	
	In particular, $F$ is finite a.e. and thus there is  a set $N$ such that $|N|=0$ and  
	$$\RR^n=\cup_{k\in\ZZ, j\geq 0} \{0<F\leq 2^{k+1} F_{j,k}\}\cup \{F=0\} \cup N.$$
	
	Thus we find
	\begin{align*}
	\beta_{\alpha,s}&\left( \sum_{k\in \ZZ} \sum_{j\geq 0}|u| \chi_{E_{j,k}}\right) \leq  A \int_{\RR^n} F^s (G_\alpha*F)^{1-s} dx	\\
	&\leq A \sum_{k\in\ZZ} \sum_{j\geq 0} \int_{\{0<F\leq 2^{k+1} F_{j,k}\}} F^s (G_\alpha*F)^{1-s} dx\\
	&\leq A \sum_{k\in\ZZ} \sum_{j\geq 0} \int_{\RR^n} (2^kF_{j,k})^s (G_\alpha*(2^k F_{j,k}))^{1-s} dx\\
	&\leq A \sum_{k\in\ZZ} \sum_{j\geq 0} 2^k {\rm Cap}_{\alpha,s}(E_{jk})\leq C \int_{\RR^n} |u| d {\rm Cap}_{\alpha,s}.
	\end{align*}
	
	Inequality \eqref{nonlinearnorm} now follows from \eqref{betasum} and the last bound, which completes the proof of the theorem.
\end{proof}

\begin{remark} For Riesz potentials $I_\alpha *f$ and Riesz capacities ${\rm cap}_{\alpha, s }$, where $\alpha\in(0,n)$ and $s>1$, the corresponding bound \eqref{nonlinearnorm}
	can be obtained using  \eqref{funequN} and the pointwise bound 
	\begin{equation}\label{VWh}
	(I_{\alpha}*f)^s \leq A I_{\alpha}*[f (I_{\alpha}*f)^{s-1}], 
	\end{equation}	 
	which holds for all nonnegative measurable functions $f$ (see \cite{VW, Ver}). Indeed,	for any $f\geq 0$ such that $I_\alpha*f\geq |u|^{\frac{1}{s}}$ q.e., by  
	\eqref{VWh} we have $A I_{\alpha}*[f (I_{\alpha}*f)^{s-1}] \geq  |u|$ q.e., and thus again by \eqref{VWh}, 
	\begin{align*}
	\beta_{\alpha,s}(u) &\leq A \int_{\RR^n} f^s (I_{\alpha}*f)^{(s-1)s} I_\alpha*[f (I_{\alpha}*f)^{s-1}]^{1-s} dx\\
	&\leq A \int_{\RR^n} f^s dx.
	\end{align*}
	
	Minimizing over such $f$ and recalling \eqref{funequN}, we get the corresponding bound \eqref{nonlinearnorm} as desired.
\end{remark}

\section{Proof of Theorem \ref{Maximal}}
\begin{proof}[Proof of Theorem \ref{Maximal}]
	By Theorem \ref{Newnorm2}, we have 
	$$\int_{\RR^n}|f|^q d {\rm Cap}_{\alpha,s} \simeq \inf\left\{\int_{\RR^n} h^s (G_\alpha*h)^{1-s} dx: h\geq0, (G_{\alpha}* h)^{\frac{1}{q}}\geq |f| \text{ q.e.} \right\}.$$
	
	On the other hand, for any	$h\geq0$ and  $(G_{\alpha}* h)^{\frac{1}{q}}\geq |f|$  q.e. by  Theorem 3.1 in \cite{OP} we have 
	$${\bf M}^{\rm loc} f \leq {\bf M}^{\rm loc} [(G_{\alpha}* h)^{\frac{1}{q}}]\leq A (G_{\alpha}* h)^{\frac{1}{q}} $$
	pointwise everywhere, provided $q>(n-\alpha)/n$. Thus
	\begin{align*}
	&\int_{\RR^n}|f|^q d {\rm Cap}_{\alpha,s} \\
	&\geq c\,  \inf\left\{\int_{\RR^n} g^s (G_\alpha*g)^{1-s} dx: g\geq0, (G_{\alpha}* g)^{\frac{1}{q}}\geq {\bf M}^{\rm loc} f  \text{ q.e.} \right\}\\
	&\simeq \int_{\RR^n} ({\bf M}^{\rm loc} f)^q d {\rm Cap}_{\alpha,s}.
	\end{align*}
	
	This completes the proof of the theorem.
\end{proof}		

\vspace{.2in}
\noindent {\bf Acknowledgements.} N.C. Phuc is supported in part by Simons Foundation, award number 426071.


\vspace{.2in}
\begin{thebibliography}{xxxxxx}
	
	
	
	
	
	\bibitem[Ad1]{Ad1} D. R. Adams, {\it A note on the Choquet integrals with respect to Hausdorff capacity}, Function spaces
	and applications, Proc. Lund 1986, Lecture Notes in Math. {\bf 1302} (Springer, Berlin, 1988) 115--124.
	
	\bibitem[Ad2]{Ad2} D. R.  Adams, {\it Choquet integrals in potential theory}, Publ. Mat. {\bf 42} (1998), no. 1, 3--66.
	
	\bibitem[AH]{AH} D. R. Adams  and   L. I. Hedberg, {\it Function Spaces and Potential
		Theory}, Springer-Verlag, Berlin, 1996.
	
	
	
	\bibitem[AX]{AX} D. R. Adams and J.  Xiao,  {\it Nonlinear analysis on Morrey spaces and their capacities}, Indiana Univ. Math. J. {\bf 53} (2004), 1629--1663.
	
	
	
	
	\bibitem[BP]{BP} P. Baras and M. Pierre, {\it Crit\`ere d'existence de solutions positives pour des \'equations semilin\'eaires non monotones}, Ann. Inst. H. Poincar\'e, 
	Analyse Non Lin\'eaire {\bf 2} (1985), 185--212.
	
	%
	%
	%
	%
	%
	%
	%
	%
	
	
	
	\bibitem[KV]{KV} N. J. Kalton and I. E. Verbitsky,  {\it Nonlinear equations and weighted norm inequalities}, Trans. Amer. Math. Soc. {\bf 351} (1999), no. 9, 3441--3497.
	
	%
	%
	%
	%
	%
	%
	
	
	
	\bibitem[Maz]{Maz} V. Maz'ya, {\it Sobolev spaces with applications to elliptic partial differential equations}, Second, revised and augmented edition. Grundlehren der 
	Mathematischen Wissenschaften, vol. {\bf 342}, Springer, Heidelberg (2011), p. xxviii+866. 
	
	\bibitem[MH]{MH}	V. G.  Maz'ja and V. P. Havin,  {\it A nonlinear potential theory}, Uspehi Mat. Nauk {\bf 27} (1972), no. 6, 67--138 (in Russian). English translation: Russ. Math. Surv. {\bf 27} (1972), 71--148.
	
	
	
	
	


	
	\bibitem[MS]{MS} V. G. Maz'ya and T.O. Shaposhnikova, {\it
		Theory of Sobolev Multipliers. With Applications to Differential and Integral Operators}, Grundlehren der Mathematischen Wissenschaften, vol. {\bf 337}, Springer-Verlag, Berlin (2009), p. xiv+609.
	
	
	\bibitem[MV]{MV} V. G. Maz'ya and I. E. Verbitsky, {\it Capacitary inequalities for fractional integrals, with applications to partial differential equations and Sobolev multipliers}, Ark. Mat. {\bf 33} (1995),  81--115.
	
	%
	%
	%
	%
	
	
	\bibitem[OP]{OP} K. H. Ooi and N. C. Phuc, {\it Characterizations of predual spaces to a class of Sobolev multiplier type spaces}. Submitted for publication. Available at:	http://arxiv.org/abs/2005.04349
	
	\bibitem[OV]{OV} J. Orobitg and J.  Verdera, {\it Choquet integrals, Hausdorff content and the Hardy-
		Littlewood maximal operator}, Bull. Lond. Math. Soc. {\bf 30} (1998), 145--150.
	
	%
	%
	%
	%
	%
	%
	%
	
	
		
	
	\bibitem[Ver]{Ver} I. E. Verbitsky,  {\it Nonlinear potentials and trace inequalities}, The Maz'ya Anniversary Collection, Eds. J. Rossmann, P. Tak\'ac, and
	G. Wildenhain,   Operator Theory: Adv. Appl. {\bf 110}  (1999), 323--343.
	
	\bibitem[VW]{VW} I. E. Verbitsky and R. L. Wheeden,  {\it Weighted norm inequalities for integral operators}, Trans. Amer. Math. Soc. {\bf 350} (1998), 3371--3391. 
	
\end{thebibliography}
\end{document}